\documentclass[12pt, leqno,twoside]{article}
\usepackage{amssymb}
\usepackage{amsmath}
\usepackage{amsfonts}
\usepackage{amsthm}
\usepackage{graphicx}
\usepackage[active]{srcltx}
\usepackage{color}

\allowdisplaybreaks

\def\rr{{\mathbb R}}
\def\mr{{\mathcal R}}
\def\rn{{{\rr}^n}}

\def\fz{\infty}

\def\boz{{\Omega}}

\def\ls{\lesssim}

\def\bint{{\ifinner\rlap{\bf\kern.25em--}
\int\else\rlap{\bf\kern.45em--}\int\fi}\ignorespaces}

\def\bbint{{\ifinner\rlap{\bf\kern.25em--}
\hspace{0.078cm}\int\else\rlap{\bf\kern.45em--}\int\fi}\ignorespaces}

\def\wp{{W^{1,p}(\boz)}}

\long\def\colred#1\endred{{\color{red}#1}}
\long\def\colgreen#1\endgreen{{\color{green}#1}}
\long\def\colmagenta#1\endmagenta{{\color{magenta}#1}}
\long\def\colblue#1\endblue{{\color{blue}#1}}
\long\def\colyellow#1\endyellow{{\color{yellow}#1}}

\def\r{\right}
\def\lf{\left}

\newtheorem{thm}{Theorem}[section]
\newtheorem{lem}{Lemma}[section]
\newtheorem{prop}{Proposition}[section]

\newtheorem{defn}{Definition}[section]

\numberwithin{equation}{section}

\begin{document}

\arraycolsep=1pt

\title{\Large\bf   Sobolev extensions via reflections }
\author{ Pekka Koskela and Zheng Zhu \thanks{The research of both authors has been supported by the Academy of 
Finland Grant number 323960.
Zheng Zhu was also support by the CSC grant CSC201506020103 from China.}}
\date{}
\maketitle
\begin{abstract}
We show that the extension results by Maz'ya and Poborchi for polynomial  cusps can be realized via composition operators generated by reflections. We also study the case of the complementary domains.
\end{abstract}

\section{Introduction}
A domain $\boz\subset\rr^n$ is called a $(p, q)$-extension domain, $1\leq q\leq p\leq\fz$, if every $u\in W^{1,p}(\boz)$ has an extension $Eu\in W^{1,q}_{\rm loc}(\rn)$ with $||Eu||_{W^{1,q}(\rn \setminus \overline \Omega)}\le C ||u||_{W^{1,p}(\Omega)}.$
A Lipschitz domain $\boz$ is a $(p, p)$-extension domain for all $1\le p\le\infty$ by results due to Calder\'on and Stein \cite{stein}. Jones generalized this result to a much larger class of domains, so-called $(\epsilon, \delta)$-domains, but general domains are not necessarily extension domains for any $p, q$. For example, in \cite{Mazya1, Mazya2, Mazya3}, Maz'ya and Poborchi investigated in detail a typical case where the above extension property fails: the case of a domain with an outward peak, also see \cite{Mazya, poborchi} for related results. Once a polynomial degree of the peak was fixed, they found the optimal $p, q$ for the $(p, q)$-extendability.

The idea of using reflections to construct extension operators is implicit in the results for Lipschitz domains. Gol'dshtein, Latfullin and Vodop'yanov initiated the systematic use of reflections for constructing extension operators in the Euclidean plane $\rr^2$ in \cite{GLV,GV1}. In \cite{GS}, Gol'dshtein and Sitnikov showed that the Sobolev extendability for planar outward and inward cuspidal domains of polynomial order can be achieved by a bounded linear extension operator induced by reflections. Very recently, Koskela, Pankka and Zhang \cite{KPZ} proved that for every planar Jordan $(p, p)$-extension domain with $1<p<\fz$, there exists a reflection over the boundary $\partial\boz$ which induces a bounded linear extension operator from $W^{1,p}(\boz)$ to $W^{1,p}(\rr^2)$.

In this paper, we study the Sobolev extension via reflections on outward cuspidal domains in the Euclidean space $\rn$ with $n\geq 3$. From now on, we alway assume $n\geq3$.

We distinguish a horizontal coordinate axis in $\rn$,
\begin{equation}
\rn=\rr\times\rr^{n-1}=\{z:=(t ,x): t\in\rr\ {\rm and}\ x=(x_1, \cdots, x_{n-1})\in\rr^{n-1}\}.\nonumber
\end{equation}
Let us consider the model case of $\boz^s$, the outward cuspidal domain with the degree $1<s<\fz$, defined by setting 
\begin{equation}\label{outer cuspidal}
\boz^s:=\lf\{(t, x)\in\rr\times\rr^{n-1}=\rr^n:0<t\leq1, |x|<t^s\r\}\cup B((2,0),\sqrt2).
\end{equation}
See Figure \ref{fig:3}. For the case of this model domain, the results due to Maz'ya and Poborchi state that there exists a bounded linear extension operator $E_1$ from $W^{1,p}(\boz^s)$ to $W^{1,q}(\rr^n)$, whenever $\frac{1+(n-1)s}{n}<p<\fz$ and $1\leq q<\frac{np}{1+(n-1)s}$, and there exists another bounded linear extension operator $E_2$ from $W^{1,p}(\boz^s)$ to $W^{1,q}(\rr^n)$, whenever $\frac{1+(n-1)s}{2+(n-2)s}<p<\fz$ and $1\leq q<\frac{p+(n-1)sp}{1+(n-1)s+(s-1)p}$. For $p=\frac{1+(n-1)s}{n}$, one has
\[ \frac{np}{1+(n-1)s}=\frac{p+(n-1)sp}{1+(n-1)s+(s-1)p}=n-1.\] 
Hence both $E_1$ and $E_2$ extend functions in $W^{1, \frac{1+(n-1)s}{n}}(\boz^s)$ to $W^{1, q}(\rn)$, whenever $1\leq q<n-1$. However, surprisingly, Maz'ya and Poborchi also constructed a bounded linear extension operator $E_3$ from $W^{1, \frac{(n-1)+(n-1)^2s}{n}}(\boz^s)$ to $W^{1, n-1}(\rr^n)$. All these results are sharp, see \cite{Mazya} and references therein.
 For a detailed exposition of these results, see \cite{Mazya}. 
Interestingly, the given 
extension operators for the domain $\boz^s$ above are 
linear and the formulas defining the operators do not depend on $p$ 
once $s$ and the range of $p$ are fixed. Our main result explains this phenomenon.

\begin{thm}\label{main resu}
Let $\boz^s\subset\rr^n$ be an outward cuspidal domain with the degree $s>1$. Then\\
 $(1):$ There exists a reflection $\mathcal{R}_1:\widehat{\rr^n}\to\widehat{\rr^n}$ over $\partial\boz^s$ which induces a bounded linear extension operator from $W^{1,p}(\boz^s)$ to $W^{1,q}(\rr^n)$, whenever $\frac{1+(n-1)s}{n}<p<\fz$ and $1\leq q<\frac{np}{1+(n-1)s}$.\\ 
 $(2):$ There exists another reflection $\mathcal \mr_2:\widehat{\rr^n}\to\widehat{\rr^n}$ over $\partial\boz^s$ which induces a bounded linear extension operator from $W^{1,p}(\boz^s)$ to $W^{1,q}(\rr^n)$, whenever $\frac{1+(n-1)s}{2+(n-2)s}<p<\fz$ and $1\leq q<\frac{(1+(n-1)s)p}{1+(n-1)s+(s-1)p}$.
\end{thm}

Theorem \ref{main resu} implies that both reflections $\mr_1$ and $\mr_2$ induce a bounded linear extension operator from $W^{1, \frac{(n-1)+(n-1)^2s}{n}}(\boz^s)$ to $W^{1,q}(\rr^n)$, whenever $1\leq q<n-1$. We would like to know if there exists a further reflection $\mr_3$ which induces a bounded linear extension operator from $W^{ 1, \frac{(n-1)+(n-1)^2s}{n}}(\boz^s)$ to $W^{1, n-1}(\rn)$.

 In general, we say that a reflection $\mr:\widehat{\rn}\rightarrow\widehat{\rn}$ over $\partial\boz$, for a bounded domain $\Omega$ (whose 
boundary has volume zero) induces
a bounded linear extension operator from $W^{1,p}(\boz)$ to $W^{1,q}(\rr^n)$ if there is an open set $U$ containing $\partial\boz$ so that, for every $u\in W^{1,p}(\Omega),$ 
the function $v$ defined by setting
$v=u$ on $\Omega\cap U$ and $v=u\circ \mr$ on $U\setminus\overline{\boz}$ has a representative that belongs to $W^{1,q}(U)$ with 
\begin{equation} \label{uu1}
\|v\|_{W^{1,q}(U)}\leq C\|u\|_{W^{1,p}(U\cap\boz)},
\end{equation}
for some positive constant $C$ independent of $u$. Similarly, we say that the reflection $\mr$ induces a bounded linear extension operator from $W^{1,p}(\rr^n\setminus\overline{\boz})$ to $W^{1,q}(\rr^n)$, if for every $u\in W^{1,p}(\rr^n\setminus\overline{\boz})$ the function $\tilde v$ defined by setting $\tilde v=u$ on $U\setminus\overline \boz$ and $\tilde v=u\circ \mr$ on $U\cap \boz$ has a representative that belongs to $W^{1,q}(U)$ with
\begin{equation}\label{uu2}
\|\tilde v\|_{W^{1,q}(U)}\leq C\|u\|_{W^{1,p}(U\setminus\overline{\boz})}.
\end{equation}
 Here the introduction of the open set $U$ is a 
convenient way to overcome the non-essential difficulty that functions in
$W^{1,p}(G)$ do not necessarily belong to $W^{1,q}(G)$ when $1\le q<p<\fz$ and $G$ has infinite volume. It follows from the assumption \eqref{uu1} (or \eqref{uu2}) via the use of a suitable cut-off function that $\Omega$ (or $\rr^n\setminus\overline{\boz}$, respectively) is a $(p, q)-$extension domain with a bounded linear extension operator. For this, see Section $2$.
 
The crucial point behind Theorem \ref{main resu} is that we obtain Sobolev estimates on 
$u\circ\mr$ in terms of the data on $u$. There
is a rather long history of such results, for example see \cite{Ukhlov2, Ukhlov3, Hencl, Ukhlov} and references therein. In the setting of our problem, the most relevant reference is the paper \cite{Ukhlov} by Ukhlov. What we find surprising in 
our situation is that a single $\mr_1$ induces the best bounded linear extension operator for all values $\frac{(n-1)+(n-1)^2s}{n}<p<\fz$ and another single $\mr_2$ induces the best bounded linear extension operator for all values $\frac{1+(n-1)s}{2+(n-2)s}<p<\frac{(n-1)+(n-1)^2s}{n}$, but neither $\mr_1$ nor $\mr_2$ can induce a best linear extension operator for $p=\frac{(n-1)+(n-1)^2s}{n}$. In the case of compositions from $W^{1,p}$ to $W^{1,p},$ the 
relevant estimate is
\begin{equation}\label{maar}
|D\mr(z)|^p\le C|J_\mr(z)|
\end{equation}
almost everywhere, which for $p=n$ is the pointwise condition of quasiconformality. Mappings satisfying (\ref{maar}) with $p\neq n$ apparently appeared for the first time in the works of Gehring \cite{Gehring1} and of Maz'ya \cite{Mazya4}, independently.
With some work one can show that \eqref{maar} implies the corresponding 
inequality with $p$ replaced by $q$ when either $q>p>n$ or $1\le q<p<n,$ but
not in other cases. On the other 
hand, for $n-1<p<\fz$, a result in \cite{GS} shows that \eqref{maar} together with $W^{1,p}$-regularity of $\mr$
implies the dual estimate
\begin{equation}\label{maar1}
|D\mr^{-1}(z)|^{\frac{p}{p+1-n}}\le C'|J_{\mr^{-1}}(z)|.
\end{equation}
This kind of duality actually also holds for compositions from $W^{1,p}$ to $W^{1,q}$ with $q<p$, see \cite{Ukhlov}. Also see \cite{HKARMA,Ukhlov,Vodop3} for general results on the regularity of $\mr^{-1}$.

\begin{figure}[htbp]
\centering
\includegraphics[width=0.6\textwidth]
{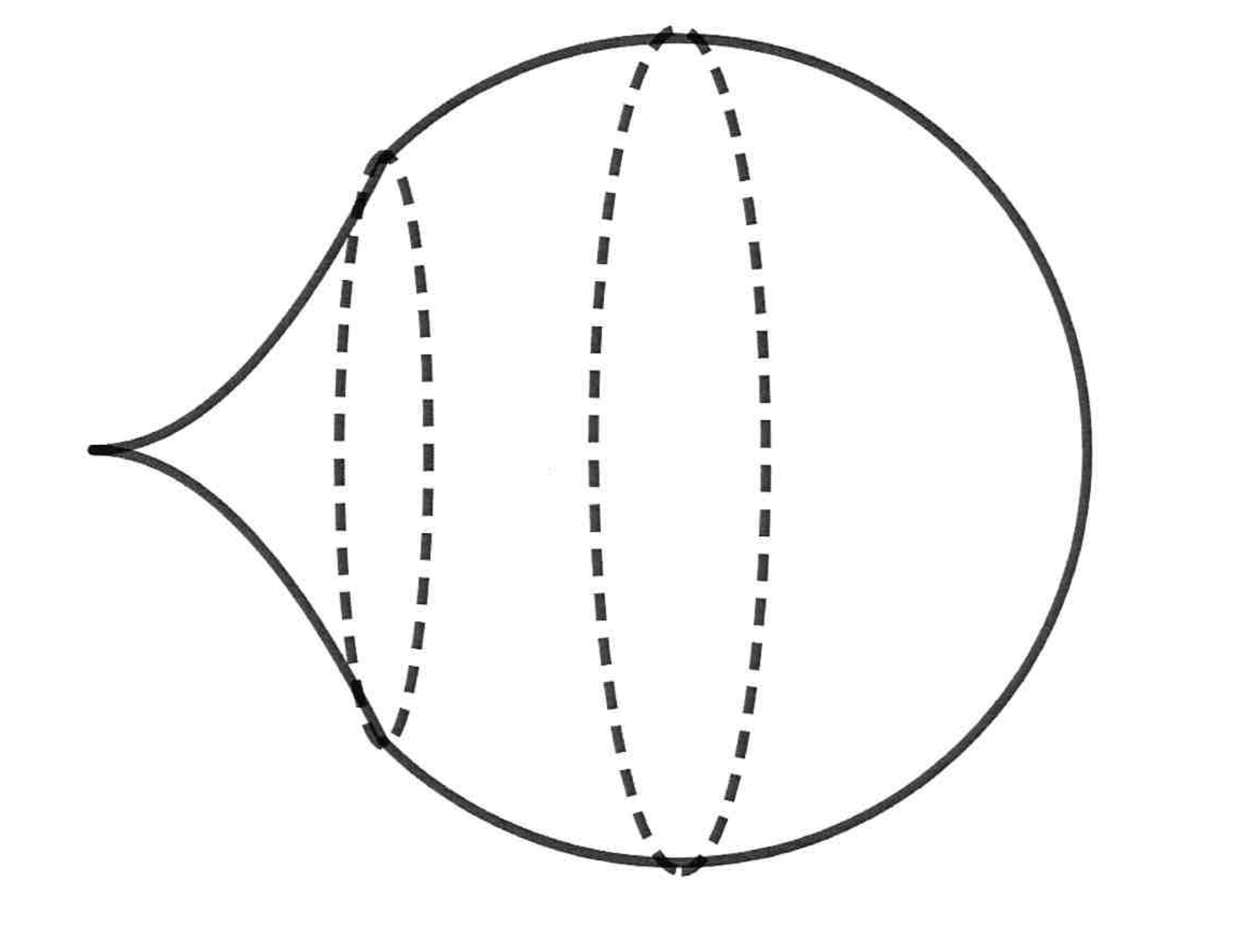}\caption{$\boz^s$}\label{fig:3}
\end{figure}

From the argument above, one could also expect that the reflections $\mr_1$ and $\mr_2$ induce a bounded linear extension operator from $W^{1,p}(\rr^n\setminus\overline{\boz^s})$ to $W^{1,q}(\rr^n)$, for some $1\leq q\leq p<\fz$. As one can easily check, for every $1<s<\fz$, $\rr^n\setminus\overline{\boz^s}$ is a so-called $(\epsilon, \delta)$-domain and hence a $(p, p)$-extension domain for every $1\leq p<\fz$ due to Jones \cite{Jones}. Our next theorem relates this to our reflections.
\begin{thm}\label{thm:comp}
For every $1<s<\fz$, $\rr^n\setminus\overline{\boz^s}$ is a $(p, p)$-extension domain, for every $1\leq p<\fz$. The reflection $\mr_1$ over $\partial\boz^s$ in Theorem \ref{main resu} induces a bounded linear extension operator from $W^{1,p}(\rr^n\setminus\overline{\boz^s})$ to $W^{1,p}(\rr^n)$, whenever $1\leq p\leq n-1$. Moreover, for each $n-1<p<\fz$, no reflection over $\partial\boz^s$ can induce a bounded linear extension operator from $W^{1,p}(\rr^n\setminus\overline{\boz^s})$ to $W^{1,p}(\rr^n)$.
\end{thm}

What then about the case $p=\fz$? We say that a domain $\boz\subset\rr^n$ is uniformly locally quasiconvex if there exist constants $C>0$ and $R>0$ such that for every pair of points $x, y\in\boz$ with $d(x, y)<R$, there is a rectifiable curve $\gamma$ connecting $x$ and $y$ in $\boz$ such that the length of $\gamma$ is bounded from above by $Cd(x, y)$. If the above holds without the distance restriction, $\boz$ is said to be quasiconvex. Recall that $\boz$ is an $(\fz, \fz)$-extension domain if and only if it is uniformly locally quasiconvex, see \cite{HKT} by Haj\l{}asz, Koskela and Tuominen. One can easily check that both $\boz^s$ and $\rr^n\setminus\overline{\boz^s}$ are uniformly locally quasiconvex, equivalently, they are $(\fz, \fz)$-extension domains. We close this introduction with the following analog of Theorem \ref{thm:comp}.
\begin{thm}\label{thm:infty}
Given $1<s<\fz$, both $\boz^s$ and $\rr^n\setminus\overline{\boz^s}$ are $(\fz, \fz)$-extension domains. The reflection $\mr_1$ over $\partial\boz^s$ in Theorem \ref{main resu} induces a bounded linear extension operator from $W^{1,\fz}(\boz^s)$ to $W^{1,\fz}(\rr^n)$. On the other hand, no reflection over $\partial\boz^s$ can induce a bounded linear extension operator from $W^{1,\fz}(\rr^n\setminus\overline{\boz^s})$ to $W^{1,\fz}(\rr^n)$.
\end{thm} 

\section{Preliminaries}
In this paper, $\widehat{\rr^n}:=\rr^n\cup\{\fz\}$ is the one-point compactification of $\rr^n$. Next, $z=(t, x)\in\rr\times\rr^{n-1}=\rr^n$ means a point in the $n$-dimensional Euclidean space $\rn$. We write $C=C(a_{1},a_{2},...,a_{n})$ to indicate a constant $C$ that depends only on the parameters $a_{1},a_{2},...,a_{n}$; the notation $A\ls B$ means there exists a finite constant $c$ with $A\le cB$ , and $A\sim_{c} B$ means $\frac{1}{c}A\leq B\leq cA$ for a constant $c>1$. Typically $c, C,...$ will be constants that depend on various parameters and may differ even on the same line of inequalities. The Euclidean distance between given points $z_1, z_2$ in Euclidean space $\rr^n$ is denoted by $d(z_1, z_2)$ or $|z_1-z_2|$. Then the distance between two sets $A, B\subset\rn$ is denoted by 
\[d(A, B):=\inf\{d(z_1, z_2): z_1\in A, z_2\in B\}.\] The open ball of radius $r$ centered at the point $z$ is denoted by $B(z, r)$. In what follows, $\Omega\subset\rr^n$ is always a domain, and $\partial\boz$ is the boundary of $\boz$.  The $r$-neighborhood of $\Omega$ is 
\[B(\Omega,r):=\{z\in\rr^n: d(z, \Omega)<r\}.\] 
Given a Lebesgue measurable set $A\subset\rr^n$, $|A|$ refers to the $n$-dimensional Lebesgue measure. The interior of a set $A\subset\rn$ is denoted by $\mathring A$. For a locally integrable function $u$ and a measurable set $A\subset\rr^n$ with $0<|A|<\fz$, we define the integral average of $u$ over $A$ by setting
\begin{equation}
\bint_Au(z)dz:=\frac{1}{|A|}\int_Au(z)dz.\nonumber
\end{equation}



The Sobolev space $W^{1,p}(\Omega)$ for $p\in[1,\infty]$ is the collection of all functions $u\in L^p(\Omega)$ whose norm
\begin{equation}
\|u\|_{W^{1,p}(\Omega)}:=\|u\|_{L^p(\Omega)}+\||D u|\|_{L^{p}(\Omega)}\nonumber
\end{equation}
is finite. Here $D u=(g_{1},g_{2},...,g_{n})$ is the $distributional\ gradient$ of $u$, where $g_{i}$ is the weak partial derivative of $u$ with respect to $x_{i}$. A mapping $f=(f_1, f_2, \cdots, f_m):\boz\to\boz'$ is said to be in the class $W^{1,p}(\boz, \boz')$, if every component $f_i$ is in the Sobolev space $W^{1, p}(\boz)$.


The outward cuspidal domain $\boz^s$ has a boundary singularity but it is still rather nice. For example, both the outward cuspidal domain $\boz^s$ and its complement $\rn\setminus\overline{\boz^s}$ satisfy the segment condition. 
\begin{defn}\label{defn:segment}
We say that a domain $\boz\subset\rn$ satisfies the segment condition if every $x\in\partial\boz$ has a neighborhood $U_x$ and a nonzero vector $y_x$ such that if $z\in\overline\boz\cap U_x$, then $z+ty_x\in\boz$ for $0<t<1$.
\end{defn}
For a domain satisfying the segment condition, we have the following lemma. See \cite[Theorem 3.22]{adams}.
\begin{lem}\label{lem:density}
If the domain $\boz\subset\rn$ satisfies the segment condition, then the set of restrictions to $\boz$ of functions in $C_o^\fz(\rn)$ is dense in $W^{1,p}(\boz)$ for $1\leq p<\fz$. In short, $C_o^\fz(\rn)\cap W^{1,p}(\boz)$ is dense in $W^{1,p}(\boz)$ for $1\leq p<\fz$.
\end{lem}

Let us give the definition of Sobolev extension domains.
\begin{defn}\label{(p,q)-extension}
Let $1\leq q\leq p\leq\fz$. We say that a domain $\boz\subset\rr^n$ is a $(p,q)$-extension domain, if for every $u\in W^{1,p}(\boz)$, there exists a function $Eu\in W_{\rm loc}^{1,q}(\rr^n)$ with $Eu\big|_{\boz}\equiv u$ and 
\begin{equation}
\|Eu\|_{W^{1,q}(\rr^n\setminus\overline{\boz})}\leq C\|u\|_{\wp}\nonumber
\end{equation}
with a constant $C$ independent of $u$. 
\end{defn}

Lipschitz domains are typical examples of Sobolev extension domains. By the results due to Calder\'on and Stein \cite{stein}, Lipschitz domains are $(p, p)$-extension domains for $1\leq p\leq\fz$. 
For the definition of Lipschitz domains, please see \cite[Definition 4.4]{Evans}. As a generalization of the extension result for Lipschitz domains, Jones \cite{Jones} proved that $(\epsilon, \delta)$-domains are also $(p, p)$-extension domains. 
\begin{defn}\label{defn:ED}
We say $\boz\subset\rn$ is an $(\epsilon, \delta)$-domain for some positive constant $0<\epsilon<1$ and $\delta>0$ if whenever $z_1, z_2\in\boz$ with $|z_1-z_2|<\delta$, there is a rectifiable arc $\gamma\subset\boz$ joining $x$ to $y$ and satisfying 
$$l(\gamma)\leq\frac{1}{\epsilon}|z_1-z_2|$$
and
$$d(z, \boz^c)\geq\frac{\epsilon|z_1-z||z_2-z|}{|z_1-z_2|}\ {\rm for\ all}\ z\ {\rm on} \gamma.$$
\end{defn}
\begin{defn}\label{defn:ref}
Let $\boz\subset\rr^n$ be a domain. A self-homeomorphism $\mr:\widehat{\rr^n}\to\widehat{\rr^n}$ is called a reflection over $\partial\boz$, if $\mr(\widehat{\rr^n}\setminus\overline\boz)=\boz$, $\mr(\boz)=\widehat{\rr^n}\setminus\overline\boz$ and for every $z\in\partial\boz$, $\mr(z)=z$. 
\end{defn}


The following technical lemma justifies our terminology.
\begin{prop}\label{cut-off}
Let $\boz\subset\rn$ be a bounded domain with $|\partial\boz|=0$ and $\mr:\widehat{\rn}\to\widehat{\rn}$ be a reflection over $\partial\boz$. If $\mr$ induces a bounded linear extension operator from $W^{1,p}(\boz)$ to $W^{1,q}(\rn)$ in the sense of (\ref{uu1}) $($from $W^{1,p}(\rn\setminus\overline\boz)$ to $W^{1,q}(\rn)$, respectively$)$ for $1\leq q\leq p<\fz$, then $\boz$ ($\rn\setminus\overline\boz$, respectively) is a $(p, q)$-extension domain with a linear extension operator. 
\end{prop}
\begin{proof}
We only consider the case of $\boz$, since the case of $\rn\setminus\overline\boz$ is analogous. Let $U\subset\rn$ be the corresponding open set which contains $\partial\boz$. For a given function $u\in W^{1,p}(\boz)$, we define a function $E_{\mr}(u)$ by setting
\begin{equation}
E_{\mr}(u)(z):=\left\{\begin{array}{ll}\label{equa:E_r(u)}
u(\mr(z)),&\ {\rm for}\ z\in U\setminus\overline{\boz},\\
0,&\ {\rm for}\ z\in\partial\boz,\\
u(z),&\ {\rm for}\ z\in \Omega.
\end{array}\right.
\end{equation}
Then $E_\mr(u)$ has a representative that belongs to $W^{1,q}(U)$ with 
$$\|E_\mr(u)\|_{W^{1,q}(U)}\leq C\|u\|_{W^{1,p}(\boz)}.$$
Let $\psi:\rn\to\rr$ be a Lipschitz function such that $\psi\big|_{\overline\boz}\equiv1$, $\psi\big|_{\rn\setminus U}\equiv 0$ and $0\leq \psi(z)\leq 1$ for every $z\in\rn$.  For every function $u\in W^{1,p}(\boz)$, we define a function on $\rn$ by setting 
\begin{equation}\label{equa:exglo}
\tilde E_\mr(u):=\psi\cdot E_\mr(u).
\end{equation}
Since $\psi$ is Lipschitz with $0\leq \psi\leq 1$, $\tilde E_\mr(u)$ has a representative that belongs to $W^{1,p}(\rn)$. Now
\begin{eqnarray}
\int_{\rr^n}|\tilde E_\mr(u)(z)|^{q}dz&\leq&\int_{\Omega}|u(z)|^{q}dz+\int_U|E_\mr(u)(z)|^qdz\nonumber\\
                                                       &\leq&\lf(\int_{\boz}|u(z)|^pdz+\int_\boz|Du(z)|^pdz\r)^{\frac{q}{p}},\nonumber
\end{eqnarray}
and 
\begin{eqnarray}
\int_{\rr^n}|D\tilde E_{\mr}(u)|^{q}dz&\leq&C\int_{U}|E\mr(u)D\psi|^qdz+C\int_{U}|\psi\nabla E_\mr(u)|^qdz\nonumber\\
                                                   & &+C\int_\boz|Du|^qdz\nonumber\\
                                         &\leq&C\lf(\int_{\Omega}|u|^pdz+\int_{\Omega}|\nabla u|^{p}dz\r)^{\frac{q}{p}}.\nonumber
\end{eqnarray}
By combining these two inequalities, we obtain that $\tilde E_\mr(u)\in W^{1,q}(\rn)$ with $\tilde E_\mr(u)\big|_\boz\equiv u$ and 
$$\|\tilde E_\mr(u)\|_{W^{1,q}(\rn)}\leq C\|u\|_{W^{1,p}(\boz)}.$$
Hence, we defined a bounded linear extension operator from $W^{1,p}(\boz)$ to $W^{1,q}(\rn)$ in (\ref{equa:exglo}).
\end{proof}
By Proposition \ref{cut-off}, in order to prove that a reflection $\mr$ over $\partial\boz^s$ can induce a bounded linear extension operator from $W^{1,p}(\boz)$ to $W^{1,q}(\rn)$ for some $1\leq q\leq p\leq\fz$, it suffices to prove that for every $u\in W^{1,p}(\boz)$, the function $E_\mr(u)$ defined in (\ref{equa:E_r(u)}) satisfies the inequality
\[\|E_\mr(u)\|_{W^{1,q}(U)}\leq C\|u\|_{W^{1,p}(\boz)}\]
with a constant $C$ independent of $u$.

Let $f:\boz\to\boz'$ be a homeomorphism. If for every $z\in U$ there is an open set containing $z$ and a constant $C>1$ such that for every $x, y\in U$, we have 
\[\frac{1}{C}|x-y|\leq |f(x)-f(y)|\leq C|x-y|,\]
we call it a locally bi-Lipschitz homeomorphism.

By combining results in \cite{Ukhlov, Vodop1, Vodop2, Vodop}, we obtain following two lemmas.
\begin{lem}\label{QCcompo}
Suppose that $f:\boz\to\boz'$ is a homeomorphism in the class $W^{1,1}_{\rm loc}(\boz, \boz')$. Fix $1\leq p<\fz$. Then the following assertions are equivalent:\\
$(1):$ for every locally Lipschitz function $u,$ the inequality
\[\int_{\boz}|D(u\circ f)(z)|^pdz\leq C\int_{\boz'}|Du(z)|^pdz\]
holds for a positive constant $C$ independent of $u$;\\
$(2):$ the inequality 
\[|Df(z)|^p\leq C(p)|J_f(z)|\]  
holds almost everywhere in $\boz$.
\end{lem} 
\begin{lem}\label{lem:pQc}
Let $1\leq q<p<\fz$. Suppose that $f:\boz\to\boz'$ is a homeomorphism in the class $W^{1, 1}_{\rm loc}(\boz, \boz')$. Then the following assertions are equivalent:\\
$(1):$ for every locally Lipschitz function $u$, the inequality 
\[\lf(\int_\boz|Du\circ f(z)|^qdz\r)^{\frac{1}{q}}\leq C\lf(\int_{\boz'}|Du(z)|^pdz\r)^{\frac{1}{p}}\]
holds for a positive constant $C$ independent of $u$;\\
$(2):$ \[\int_{\boz}\frac{|Df(z)|^{\frac{pq}{p-q}}}{|J_f(z)|^{\frac{q}{p-q}}}dz<\fz.\]
\end{lem}
The following lemma is a special case of \cite[Theorem 3]{Vodop3}.
\begin{lem}\label{reduinverse}
Let $\boz,\boz'\subset\rr^n$ be domains,  and let $f:\boz\rightarrow\boz'$ be a homeomorphism in the class $W_{\rm loc}^{1,p}(\boz, \boz')$ for a fixed $n-1<p<\fz$. If 
\begin{equation}\label{pdisin}
|Df(z)|^p\leq C(p)|J_f(z)|
\end{equation} 
holds for almost every $z\in\boz$, then the inverse homeomorphism $f^{-1}:\boz'\to\boz$ belongs to the class $W^{1,\frac{p}{p+1-n}}_{\rm loc}(\boz',\boz)$ with 
\begin{equation}\label{inverdisin}
|Df^{-1}(z)|^{\frac{p}{p+1-n}}\leq C(p)|J_{f^{-1}}(z)|
\end{equation}
for almost every $z\in\boz'$.
\end{lem}

\section{Main Results}
In this section, we show that the Sobolev extension results for outward cuspidal domain $\boz^s\subset\rr^n$ from \cite{Mazya1, Mazya2, Mazya3} can be achieved by bounded linear extension operators induced by reflections, except possibly for the case from $W^{1,\frac{1+(n-1)s}{n}}(\boz^s)$ to $W^{1, n-1}(\rn)$. Let us begin by introducing two reflections. 

\subsection{Reflection $\mr_1$ over $\partial\boz^s$}\label{sec:ref1}
In order to introduce the reflection $\mr_1:\widehat{\rr^n}\to\widehat{\rr^n}$ over $\partial\boz^s$, we define a domain $\Delta\subset\rr^n$ by setting
\begin{equation}\label{equa:delta}
\Delta:=\lf\{(t, x)\in\rr\times\rr^{n-1}; \frac{-1}{2}<t<\frac{1}{2}, |x|<\frac{1}{2}\r\}\cup\boz^s.
\end{equation}
 See Figure \ref{fig:5}. To begin, we divide $\Delta\setminus\overline{\boz^s}$ into three parts $A, B, C$ by setting 
\begin{center}
$A:=\lf\{(t, x)\in\rr\times\rr^{n-1}; \frac{-1}{2}<t\leq 0, |x|\leq|t|\r\}$,
\end{center}
\begin{center}
$B:=\lf\{(t, x)\in\rr\times\rr^{n-1}; \frac{-1}{2}< t<\frac{1}{2}, |t|\leq|x|<\frac{1}{2}\r\}$
\end{center}
and
\begin{center}
$C:=\lf\{(t, x)\in\rr\times\rr^{n-1}; 0\leq t<\frac{1}{2}, t^{s}\leq|x|\leq t\r\}$.
\end{center}
\begin{figure}[htbp]
\centering
\includegraphics[width=0.7\textwidth]
{Delta.png}\caption{The domain $\Delta$}\label{fig:5}
\end{figure} 
We define a subdomain $\boz^s_1\subset\boz^s$ by setting 
\begin{equation}\label{subdomain}
\boz^s_1:=\lf\{(t, x)\in\boz^s; 0<t<\frac{1}{2}, |x|<t^{s}\r\}.
\end{equation}
We will construct a reflection $\mr_1$ which maps $\Delta\setminus\overline{\boz^s}$ onto $\boz_1^s$. We define $\mr_1$ on $\Delta\setminus\overline{\boz^s}$ by setting
\begin{equation}\label{HOMEOMOR'}
\mr_1(t,x):=\left\{\begin{array}{ll}
\lf(-t, \frac{1}{6}|t|^{s-1}x\r),&\ \  {\rm if}\ \ (t, x)\in A,\\

\lf(|x|, \frac{t}{6}|x|^{s-2}x+\frac{1}{3}|x|^{s-1}x\r),&\ \  {\rm if}\ \ (t, x)\in B,\\

\lf(t, \frac{t^{s-1}}{2(t^{s-1}-1)}x+\lf(t^s-\frac{t^{2s-1}}{2(t^{s-1}-1)}\r)\frac{x}{|x|}\r),&\ \ {\rm if}\ \ (t, x)\in C.
\end{array}\right.
\end{equation}
We extend $\mr_1$ to $\partial\boz^s$ as the identity. Since both $\partial\Delta$ and $\partial(\boz^s\setminus\boz^s_1)$ are bi-Lipschitz equivalent to the unit sphere, it is easy to check that we can construct a reflection $\mr_1:\widehat{\rn}\to\widehat{\rn}$ over $\partial\boz^s$ such that $\mr_1$ is defined as above on $\Delta\setminus\boz^s$, and $\mr_1$ is bi-Lipschitz on $B(\boz^s, 1)\setminus\Delta$.

For $(t, x)\in\mathring A$, the resulting differential matrix of $\mr_1$ is 
\begin{eqnarray}\label{differeninA}
D\mr_1(t, x)
&=&
\left(
 \begin{array}{ccccc}
 -1&~~ 0 &~~ 0 &~~\cdots &~~ 0\\
\frac{1-s}{6}|t|^{s-2}x_1  &~~ \frac{1}{6}|t|^{s-1}&~~ 0 &~~ \cdots &~~0\\
\frac{1-s}{6}|t|^{s-2}x_2 &~~ 0 &~~ \frac{1}{6}|t|^{s-1}&~~ \cdots &~~0\\

\vdots &~~ \vdots &~~ \vdots &~~ \ddots &~~ 0\\
 \frac{1-s}{6}|t|^{s-2}x_{n-1}&~~ 0 &~~ \cdots &~~ 0&~~\frac{1}{6}|t|^{s-1}\\
\end{array}
\right). \\ \nonumber
\end{eqnarray}
Hence, for every $(t, x)\in\mathring A$, we have
\begin{equation}\label{equa:ref11}
|D\mr_1(t, x)|\ls1\ {\rm and}\ |J_{\mr_1}(t, x)|\sim_c |t|^{(n-1)(s-1)}.
\end{equation}

For $(t, x)\in\mathring B$, the resulting differential matrix $\mr_1$ is 
\begin{eqnarray}\label{differeninB}
D\mr_1(t, x)
&=&
\left(
 \begin{array}{ccccc}
 0&~~ \frac{x_1}{|x|} &~~ \frac{x_2}{|x|} &~~\cdots &~~ \frac{x_{n-1}}{|x|}\\
\frac{x_1}{6}|x|^{s-2}  &~~ A^1_1(t,x)&~~ A^1_2 (t, x)&~~\cdots &~~A^1_{n-1}(t, x)\\
 \frac{x_2}{6}|x|^{s-2}&~~ A^2_1(t, x) &~~ A^2_2(t, x) &~~\cdots&~~A^2_{n-1}(t, x)\\
 \vdots&~~\vdots&~~\vdots&~~\cdots&~~\vdots\\
 \frac{x_{n-1}}{6}|x|^{s-2} &~~ A^{n-1}_1(t, x) &~~ A^{n-1}_2(t, x) &~~\cdots &~~A^{n-1}_{n-1}(t, x)
\end{array}
\right), \\ \nonumber
\end{eqnarray}
where, for every $i, j\in\{1, 2,\cdots, n-1\}$, we set
 \begin{equation}
 A^i_j(t,x):=\left\{\begin{array}{ll}
\lf(\frac{t}{6}|x|^{s-2}+\frac{1}{3}|x|^{s-1}\r)+\lf(\frac{t}{6}(s-2)\frac{x_i^2}{|x|^{4-s}}+\frac{s-1}{3}\frac{x_i^2}{|x|^{3-s}}\r),&\ \  {\rm if}\ \ i=j,\\

\frac{t}{6}(s-2)\frac{x_ix_j}{|x|^{4-s}}+\frac{s-1}{3}\frac{x_ix_j}{|x|^{3-s}},&\ \ {\rm if}\ \ i\neq j.
\end{array}\right.\nonumber
 \end{equation}
Since $|t|\leq |x|<\frac{1}{2}$, a simple computation gives 
\begin{equation}\label{equa:ref12}
|D\mr_1(t, x)|\ls1\ {\rm and}\ |J_{\mr_1}(t, x)|\sim_c\sum_{k=1}^{n-1}\frac{x_k^2}{6}|x|^{s-3}\prod_{i\neq k}A^i_1\sim_c|x|^{(n-1)(s-1)}.
\end{equation} 

For $(t, x)\in\mathring C$, the resulting differential matrix of $\mr_1$ is 
\begin{eqnarray}\label{differeninC}
D\mr_1(t, x)
&=&
\left(
 \begin{array}{ccccc}
 1&~~ 0 &~~ 0&~~\cdots&~~0\\
 A^1_t(t, x) &~~ A^1_1(t, x) &~~ A^1_2(t, x)&~~\cdots &~~ A^1_{n-1}(t, x)\\
 A^2_t(t, x)&~~ A^2_1(t, x) &~~ A^2_2(t, x) &~~\cdots&~~ A^2_{n-1}(t, x)\\
 \vdots&~~\vdots&~~\vdots&~~\cdots&~~\vdots\\
 A^{n-1}_t(t, x) &~~ A^{n-1}_1(t, x) &~~ A^{n-1}_2(t, x)&~~\cdots &~~ A^{n-1}_{n-1}(t, x)
\end{array}
\right),\\ \nonumber
\end{eqnarray}
where, for every $i, j\in\{1, 2,\cdots, n-1\}$, we have 

\begin{equation}
A^i_j(t, x):=\left\{\begin{array}{ll}
\frac{t^{s-1}}{2(t^{s-1}-1)}+\lf(t^s-\frac{t^{2s-1}}{2(t^{s-1}-1)}\r)\lf(\frac{1}{|x|}-\frac{x_i^2}{|x|^3}\r),&\ \  {\rm if}\ \ i=j,\\

\lf(\frac{t^{2s-1}}{2(t^{s-1}-1)}-t^s\r)\frac{x_ix_j}{|x|^3},&\ \ {\rm if}\ \ i\neq j.
\end{array}\right.\nonumber
\end{equation} 
and
\begin{eqnarray}
A^i_t(t, x)&:=&x_i\lf(\frac{(s-1)t^{s-2}}{2(t^{s-1}-1)}-\frac{(s-1)t^{2s-3}}{2(t^{s-1}-1)^2}\r)\nonumber\\
       & &+\frac{x_i}{|x|}\lf(st^{s-1}-\frac{(2s-1)t^{2s-2}}{2(t^{s-1}-1)}+\frac{(s-1)t^{3s-3}}{2(t^{s-1}-1)^2}\r).\nonumber
\end{eqnarray}
Since $t^s<|x|<t$, a simple computation gives 
\begin{equation}\label{equa:ref13}
|D\mr_1(t, x)|\ls1\ {\rm and}\ |J_{\mr_1}(t, x)|\sim_ct^{(n-1)(s-1)}.
\end{equation}

Finally, since $\mr_1$ is bi-Lipschitz on $B(\boz^s, 1)\setminus\overline{\Delta}$, there exists a positive constant $C>1$ such that for almost every $(t, x)\in B(\boz^s, 1)\setminus\overline{\Delta}$, we have 
\begin{equation}\label{equa:upper}
\frac{1}{C}\leq |D\mr_1(t, x)|\leq C \ {\rm and}\ \frac{1}{C}\leq |J_{\mr_1}(t, x)|\leq C.\nonumber
\end{equation}

It is easy to see that the restriction of $\mr_1$ to $B(\boz^s, 1)\setminus(\boz^s\cup\{0\})$ is locally bi-Lipschitz.

\subsection{Reflection $\mr_2$ over $\partial\boz^s$}
In order to introduce the reflection $\mr_2:\widehat{\rr^n}\to\widehat{\rr^n}$ over $\partial\boz^s$, we define a domain $\Delta'\subset\rr^n$ by setting
\begin{equation}\label{cone}
\Delta':=\left\{(t, x)\in\rr\times\rr^{n-1}=\rr^{n}:\frac{-1}{2}<t<\frac{1}{2}, |x|<\left(\frac{1}{2}\right)^{s}\right\}\cup \Omega^s.
\end{equation}
See Figure \ref{fig:6}.
\begin{figure}[htbp]
\centering
\includegraphics[width=0.8\textwidth]
{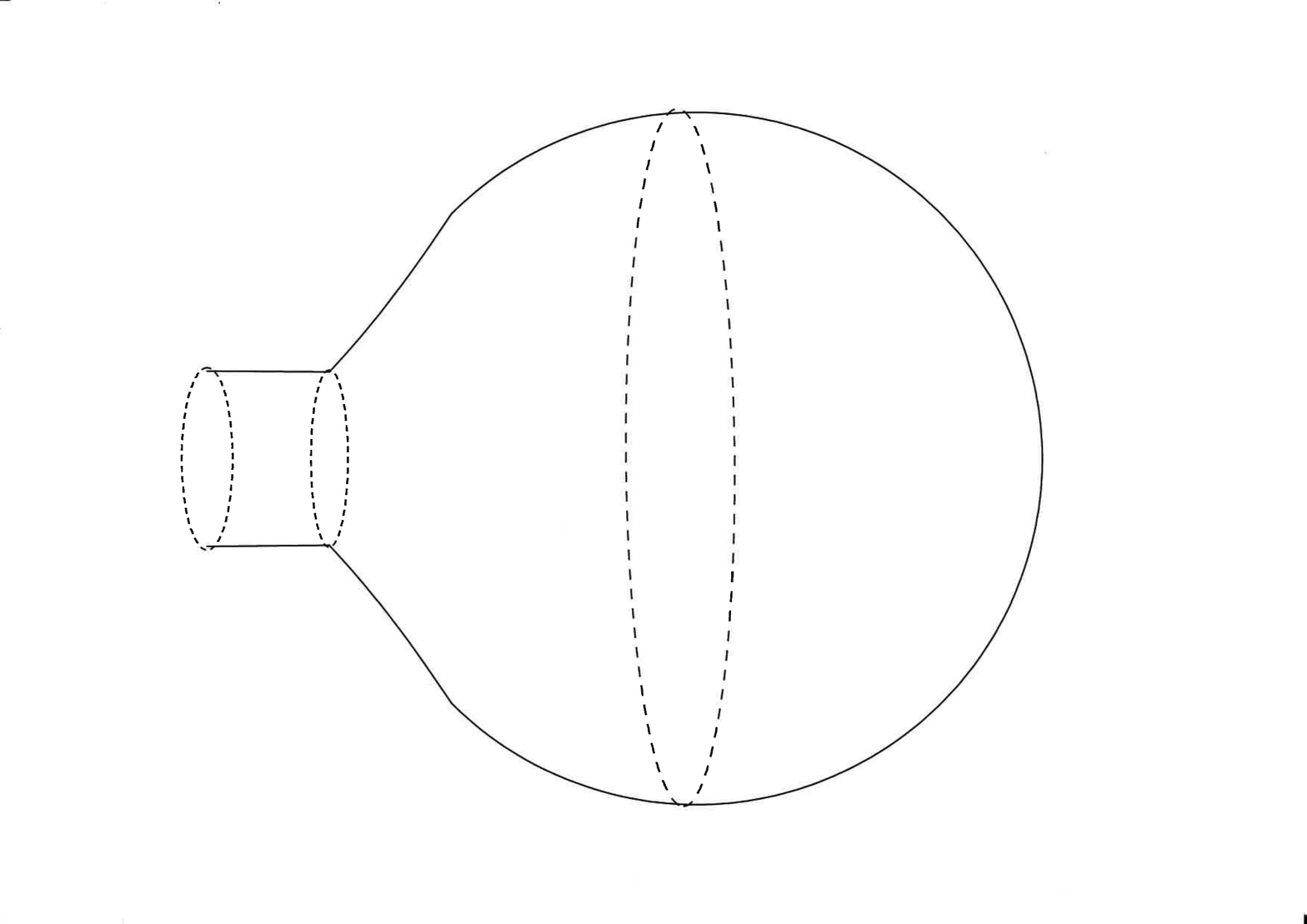}\caption{The domain $\Delta'$}\label{fig:6}
\end{figure} 
We divide $\Delta'\setminus\overline{\boz^s}$ into two parts $D, E$ by setting
\begin{equation}
D:=\lf\{(t, x)\in\rr\times\rr^{n-1}=\rr^n:\frac{-1}{2}<t\leq 0, |x|\leq|t|^{s}\r\},\nonumber
\end{equation}
and 
\begin{equation}
E:=\lf\{(t, x)\in\rr\times\rr^{n-1}=\rr^n: \frac{-1}{2}<t<\frac{1}{2}, |t|^s<|x|<\lf(\frac{1}{2}\r)^s\r\}.\nonumber
\end{equation}
We will construct a reflection $\mr_2$ over $\partial\boz^s$ which maps $\Delta'\setminus\overline{\boz^s}$ onto $\boz^s_1$. We define the reflection $\mr_2$ on $\Delta'\setminus\overline{\boz^s}$ by setting
\begin{equation}\label{HOMEOMOR1}
\mr_2(t, x):=\left\{\begin{array}{ll}
\lf(-t, \frac{1}{2}x\r),&\ \  {\rm if}\ \ (t, x)\in D,\\
\lf(|x|^{\frac{1}{s}},\frac{t}{4}\frac{x}{|x|^{\frac{1}{s}}}+\frac{3}{4}x\r),&\ \  {\rm if}\ \ (t, x)\in E.
\end{array}\right.
\end{equation}
We extend $\mr_2$ on $\partial\boz^s$ as the identity. Since both $\partial\Delta'$ and $\partial(\boz^s\setminus\boz^s_1)$ are bi-Lipschitz equivalent to the unit sphere, we can construct a reflection $\mr_2$ which is defined on $\Delta'\setminus\overline{\boz^s}$ as in (\ref{HOMEOMOR1}) and is bi-Lipschitz on $B(\boz^s, 1)\setminus\Delta'$.

For $z=(t, x)\in\mathring D$, the resulting differential matrix of $\mr_2$ is 
\begin{eqnarray}\label{differeninA1}
D\mr_2(t, x)
&=&
\left(
 \begin{array}{ccccc}
-1 &~~ 0 &~~ 0 &~~\cdots &~~0\\
0  &~~ \frac{1}{2}&~~ 0 &~~\cdots &~~0\\
0 &~~ 0 &~~ \frac{1}{2}&~~\cdots &~~0\\
\vdots &~~ \vdots &~~ \vdots &~~\ddots &~~\vdots\\
0 &~~ 0 &~~0&~~\cdots &~~\frac{1}{2}
\end{array}
\right). \\ \nonumber
\end{eqnarray}
Hence,
\begin{equation}\label{equa:ref21}
|D\mr_2|(t, x)=1\ {\rm and}\ |J_{\mr_2}(t, x)|=\frac{1}{2^{n-1}}. 
\end{equation}

For $z=(t, x)\in\mathring E$, the resulting matrix of $\mr_2$ is
\begin{eqnarray}\label{differeninB1}
D\mr_2(t, x)
&=&
\left(
 \begin{array}{ccccc}
0 &~~\frac{x_1}{s|x|^{2-\frac{1}{s}}} &~~ \frac{x_2}{s|x|^{2-\frac{1}{s}}}&~~\cdots&~~\frac{x_{n-1}}{s|x|^{2-\frac{1}{s}}}\\
\frac{x_1}{4|x|^{\frac{1}{s}}}  &~~ A^1_1(t, x)&~~ A^1_2(t, x) &~~\cdots &~~ A^1_{n-1}(t, x)\\
\frac{x_2}{4|x|^{\frac{1}{s}}} &~~ A^2_1(t, x) &~~ A^2_2(t, x) &~~\cdots &~~ A^2_{n-1}(t, x)\\
\vdots &~~\vdots &~~\vdots &~~\cdots &~~\vdots\\
\frac{x_{n-1}}{4|x|^{\frac{1}{s}}} &~~ A^{n-1}_1(t, x) &~~ A^{n-1}_2(t, x) &~~\cdots &~~ A^{n-1}_{n-1}(t, x)
\end{array}
\right), \\ \nonumber
\end{eqnarray}
where, for every $i, j\in \{1, 2, \cdots, n-1\}$, we have 
\begin{equation}
A^i_j(t, x):=\left\{\begin{array}{ll}
\frac{t}{4}\lf(\frac{1}{|x|^{\frac{1}{s}}}-\frac{x_i^2}{s|x|^{2+\frac{1}{s}}}\r)+\frac{3}{4},&\ \  {\rm if}\ \ i=j,\\
\frac{-tx_ix_j}{4s|x|^{2+\frac{1}{s}}},&\ \ {\rm if}\ \ i\neq j.
\end{array}\right.
\end{equation}
After a simple computation, for every $(t, x)\in\mathring E$ we have
\begin{equation}\label{equa:ref22}
|D\mr_2(t, x)|\ls\frac{1}{|x|^{\frac{s-1}{s}}}\ \ {\rm and}\ \ |J_{\mr_2}(t, x)|\sim_c\sum_{k=1}^{n-1}\frac{x_k^2}{4s|x|^2}\prod_{i\neq k}A^i_i=\sim_{c}1.
\end{equation}

Finally, since $\mr_2$ is bi-Lipschitz on $B(\boz^s, 1)\setminus\overline{\Delta'}$, there exists a positive constant $C>1$ such that for almost every $(t, x)\in B(\boz^s, 1)\setminus\overline{\Delta'}$, we have 
\begin{equation}\label{equa:lipschitz}
\frac{1}{C}\leq |D\mr_2(t, x)|\leq C\ {\rm and}\ \frac{1}{C}\leq |J_{\mr_2}(t, x)|\leq C.
\end{equation}
It is easy to see that the restriction of $\mr_2$ to $B(\boz^s, 1)\setminus({\boz^s\cup\{0\}})$ is locally bi-Lipschitz.

\subsection{Proof of Theorem \ref{main resu}}

We prove Theorem \ref{main resu} in two parts, considering the two reflections separately.

\begin{thm}\label{thm5}
Let $\boz^s\subset\rr^n$ be an outward cuspidal domain with the degree $s>1$. Then the reflection $\mr_1:\widehat{\rr^n}\to\widehat{\rr^n}$ over $\partial\boz^s$ induces a bounded linear extension operator from $W^{1,p}(\boz^s)$ to $W^{1,q}(\rr^n)$, whenever $\frac{1+(n-1)s}{n}<p< \fz$ and $1\leq q< \frac{np}{1+(n-1)s}$.
\end{thm}
\begin{proof}
Since $\boz^s$ satisfies the segment condition, by Lemma \ref{lem:density}, $C_o^\fz(\rn)\cap W^{1,p}(\boz^s)$ is dense in $W^{1,p}(\boz^s)$. Let $u\in C_o^\fz(\rn)\cap W^{1,p}(\boz^s)$ be arbitrary. We define a function $E_{\mr_1}(u)$ as in (\ref{equa:E_r(u)}) and another function $w$ by setting 
\begin{equation}\label{defn:funcw}
w(z):=\left\{\begin{array}{ll}
u\circ\mr_1(z),&\ \  {\rm if}\ \ z\in B(\boz^s, 1)\setminus\overline{\boz^s},\\
u(z),&\ \  {\rm if}\ \ z\in \overline{\boz^s}.
\end{array}\right.
\end{equation}
Since $u\in C_o^\fz(\rn)\cap W^{1,p}(\boz^s)$ and $\mr_1$ is locally Lipschitz on $B(\boz^s, 1)\setminus(\boz^s\cup\{0\})$, the function $w$ is locally Lipchitz on $ B(\boz^s, 1)\setminus\{0\}$. We claim that $w\in W^{1,q}(B(\boz^s, 1))$ with 
 \[\|w\|_{W^{1,q}(B(\boz^s, 1))}\leq C\|u\|_{W^{1,p}(\boz^s)}\]
 for a constant $C>1$ independent of $u$. These claims follow if we prove the above norm estimate with 
$B(\Omega^s,1)$ replaced by $B(\Omega^s,1)\setminus \{0\}.$ Next, since $w$ is locally Lipschitz and  $|\partial \Omega^s|=0,$ it suffices to estimate the norm over the union of
$\Omega^s$ and $B(\Omega^s,1)\setminus\overline{\Omega^s}.$ Since $w=u\in W^{1,p}(\Omega^s)$ on $\Omega^s,$ our domain $\Omega^s$ has finite measure and $q<p,$ we are reduced to estimating the norm over the second set in question. On this set, $w=u\circ\mr_1$ almost everywhere and hence it suffices to prove the inequality
\begin{equation}\label{equa:norm1}
\lf(\int_{B(\boz^s, 1)\setminus\overline{\boz^s}}|u\circ\mr_1(z)|^qdz\r)^{\frac{1}{q}}\leq C\lf(\int_{\boz^s}|u(z)|^pdz\r)^{\frac{1}{p}}
\end{equation}
and the inequality
\begin{equation}\label{equa:norm2}
\lf(\int_{B(\boz^s, 1)\setminus\overline{\boz^s}}|D(u\circ\mr_1)(z)|^qdz\r)^{\frac{1}{q}}\leq C\lf(\int_{\boz^s}|Du(z)|^pdz\r)^{\frac{1}{p}}.
\end{equation}
It is easy to see that
\[B(\boz^s, 1)\setminus\overline{\boz^s}=(B(\boz^s, 1)\setminus\Delta)\cup(\Delta\setminus\overline{\boz^s})\]
and $\Delta\setminus\overline{\boz^s}=A\cup B\cup C$. Since 
\[|\partial\Delta|=|\partial A|=|\partial B|=|\partial C|=0,\] 
we have 
\begin{eqnarray}\label{equa:sum1}
\int_{B(\boz^s, 1)\setminus\overline{\boz^s}}|u\circ\mr_1(z)|^qdz&=&\int_{B(\boz^s, 1)\setminus\overline{\Delta}}|u\circ\mr_1(z)|^qdz\\
       & &+\lf(\int_{\mathring A}+\int_{\mathring B}+\int_{\mathring C}\r)|u\circ\mr_1(z)|^qdz.\nonumber
\end{eqnarray}
Since $\mr_1$ is bi-Lipschitz on $B(\boz^s, 1)\setminus\overline{\Delta}$ and $|\boz^s|<\fz$, by the H\"older inequality, we have 
\begin{equation}\label{equa:sum2}
\int_{B(\boz^s, 1)\setminus\overline{\Delta}}|u\circ\mr_1(z)|^qdz\leq C\lf(\int_{\boz^s}|u(z)|^pdz\r)^{\frac{q}{p}}.
\end{equation}
By  the H\"older inequality and a change of variable, we have 
\begin{eqnarray}\label{akobi}
\int_{\mathring A}|u\circ\mr_1(z)|^qdz&\leq&\lf(\int_{\mathring A}|u\circ\mr_1(z)|^p|J_{\mr_1}(z)|dz\r)^{\frac{1}{p}}\cdot\lf(\int_{\mathring A}\frac{1}{|J^{\frac{q}{p-q}}_{\mr_1}(z)|}dz\r)^{\frac{p-q}{p}}\nonumber\\&\leq&\lf(\int_{\boz^s}|u(z)|^pdz\r)^{\frac{q}{p}}\cdot\lf(\int_{\mathring A}\frac{1}{|J^{\frac{q}{p-q}}_{\mr_1}(z)|}dz\r)^{\frac{p-q}{p}}.
\end{eqnarray}
By (\ref{equa:ref11}), we have 
\begin{equation}
\int_{\mathring A}\frac{1}{|J_{\mr_1}(z)|^{\frac{q}{p-q}}}dz\leq C\int_0^{\frac{1}{2}}t^{(n-1)-\frac{(n-1)(s-1)q}{p-q}}dt<\fz,\nonumber
\end{equation}
whenever $\frac{1+(n-1)s}{n}<p<\fz$ and $1\leq q<\frac{np}{1+(n-1)s}$. Hence, we have
\begin{equation}\label{equa:sum3}
\int_{\mathring A}|u\circ\mr_1(z)|^qdz\leq C\lf(\int_{\boz^s}|u(z)|^pdz\r)^{\frac{q}{p}}.
\end{equation}
Next, via (\ref{equa:ref12}) and (\ref{equa:ref13}), we obtain the estimates
$$\int_{\mathring B}\frac{1}{|J_{\mr_1}(z)|^{\frac{q}{p-q}}}dz\leq C\int_0^{\frac{1}{2}}|x|^{(n-1)-\frac{(n-1)(s-1)q}{p-q}}d|x|<\fz$$
and
$$\int_{\mathring C}\frac{1}{|J_{\mr_1}(z)|^{\frac{q}{p-q}}}dz\leq C\int_0^{\frac{1}{2}}x_1^{(n-1)-\frac{(n-1)(s-1)q}{p-q}}dx_1<\fz, $$ 
whenever $\frac{1+(n-1)s}{n}<p<\fz$ and $1\leq q<\frac{np}{1+(n-1)s}$. By repeating the argument leading to (\ref{equa:sum3}), we obtain the following desired analogs of (\ref{equa:sum3}):
\begin{equation}\label{equa:sum4}
\int_{\mathring B}|u\circ\mr_1(z)|^qdz\leq C\lf(\int_{\boz^s}|u(z)|^pdz\r)^{\frac{q}{p}}
\end{equation}
and
\begin{equation}\label{equa:sum5}
\int_{\mathring C}|u\circ\mr_1(z)|^qdz\leq C\lf(\int_{\boz^s}|u(z)|^pdz\r)^{\frac{q}{p}},
\end{equation}
whenever $\frac{1+(n-1)s}{n}<p<\fz$ and $1\leq q<\frac{np}{1+(n-1)s}$. Hence, (\ref{equa:norm1}) follows.

To prove inequality (\ref{equa:norm2}), by Lemma \ref{QCcompo}, it suffices to show that
 \begin{equation}
 \int_{B(\boz^s, 1)\setminus\overline{\boz^s}}\frac{|D\mr_1(z)|^{\frac{pq}{p-q}}}{|J_{\mr_1}(z)|^{\frac{q}{p-q}}}dz<\fz.\nonumber
 \end{equation}
 Clearly
 \begin{equation}
 \int_{B(\boz^s, 1)\setminus\overline{\boz^s}}\frac{|D\mr_1(z)|^{\frac{pq}{p-q}}}{|J_{\mr_1}(z)|^{\frac{q}{p-q}}}dz= \int_{B(\boz^s, 1)\setminus\overline\Delta}\frac{|D\mr_1(z)|^{\frac{pq}{p-q}}}{|J_{\mr_1}(z)|^{\frac{q}{p-q}}}dz+ \int_{\Delta\setminus\overline{\boz^s}}\frac{|D\mr_1(z)|^{\frac{pq}{p-q}}}{|J_{\mr_1}(z)|^{\frac{q}{p-q}}}dz.\nonumber
 \end{equation}
First, by inequality (\ref{equa:lipschitz}), we have
\begin{equation}
\int_{B(\boz^s, 1)\setminus\overline\Delta}\frac{|D\mr_1(z)|^{\frac{pq}{p-q}}}{|J_{\mr_1}(z)|^{\frac{q}{p-q}}}dz<\fz.\nonumber
\end{equation}
Since $\Delta\setminus\overline{\boz_s}= A\cup B\cup C$ and $|\partial A|=|\partial B|=|\partial C|=0$, we have
\begin{eqnarray}\label{Summa3}
\int_{\Delta\setminus\overline{\boz_s}}\frac{|D\mr_1(z)|^{\frac{pq}{p-q}}}{|J_{\mr_1}(z)|^{\frac{q}{p-q}}}dz&=&\lf(\int_{\mathring A}+\int_{\mathring B}+\int_{\mathring C}\r)\frac{|D\mr_1(z)|^{\frac{pq}{p-q}}}{|J_{\mr_1}(z)|^{\frac{q}{p-q}}}dz.\nonumber
\end{eqnarray}
By (\ref{equa:ref11}), (\ref{equa:ref12}) and (\ref{equa:ref13}), we obtain
\begin{equation}
\int_{\mathring A}\frac{|D\mr_1(z)|^{\frac{pq}{p-q}}}{|J_{\mr_1}(z)|^{\frac{q}{p-q}}}dz\leq C\int_0^{\frac{1}{2}}t^{(n-1)-\frac{(n-1)(s-1)q}{p-q}}dt<\fz,\nonumber
\end{equation}
\begin{equation}
\int_{\mathring B}\frac{|D\mr_1(z)|^{\frac{pq}{p-q}}}{|J_{\mr_1}(z)|^{\frac{q}{p-q}}}dz\leq C\int_0^{\frac{1}{2}}|x|^{(n-1)-\frac{(n-1)(s-1)q}{p-q}}d|x|<\fz,\nonumber
\end{equation}
and
\begin{equation}
\int_{\mathring C}\frac{|D\mr_1(z)|^{\frac{pq}{p-q}}}{|J_{\mr_1}(z)|^{\frac{q}{p-q}}}dx\leq C\int_0^{\frac{1}{2}}x_1^{(n-1)-\frac{(n-1)(s-1)q}{p-q}}dx_1<\fz,\nonumber
\end{equation}
whenever $\frac{1+(n-1)s}{n}<p<\fz$ and $1\leq q<\frac{np}{1+(n-1)s}$. In conclusion, we have proved that $w\in W^{1,q}(B(\boz^s, 1))$ with the bound
$$\|w\|_{W^{1,q}(B(\boz^s, 1))}\leq C\|u\|_{W^{1,p}(\boz^s)}$$
whenever $\frac{1+(n-1)s}{n}<p<\fz$ and $1\leq q<\frac{np}{1+(n-1)s}$. Since $E_{\mr_1}(u)=w$ almost everywhere, the above also holds with $w$ replaced by $E_{\mr_1}(u)$. 

For an arbitrary $u\in W^{1, p}(\boz^s)$, by the density of $C_o^\fz(\rn)\cap W^{1,p}(\boz^s)$, we can find a sequence of functions $\{u_i\}_{i=1}^\fz\subset C_o^\fz(\rn)\cap W^{1,p}(\boz^s)$ and a subset $N\subset\boz^s$ with $|N|=0$ such that 
\begin{equation}\label{equa:limit0}
\lim_{i\to\fz}\|u_i-u\|_{W^{1,p}(\boz^s)}=0,
\end{equation} 
and for every $z\in\boz^s\setminus N$, 
\begin{equation}\label{equa;limit}
\lim_{i\to\fz}|u_i(z)-u(z)|=0.
\end{equation}
By the argument above, for every $u_i\in C_o^\fz(\rn)\cap W^{1,p}(\boz^s)$, we have $E_{\mr_1}(u_i)\in W^{1,q}(B(\boz^s, 1))$ and 
\begin{equation}\label{equa:normc}
\|E_{\mr_1}(u_i)\|_{W^{1,q}(B(\boz^s, 1))}\leq C\|u_i\|_{W^{1,p}(\boz^s)}
\end{equation}
with a constant $C$ independent of $u_i$. Since $\mr_1$ is locally bi-Lipschitz on $B(\boz^s, 1)\setminus\overline{\boz^s}$, we have $\mr_1(N)\subset B(\boz^s, 1)\setminus\overline{\boz^s}$ with $|\mr_1(N)|=0$. By the definition of $E_{\mr_1}(u_i)$ in (\ref{equa:E_r(u)}), the sequence $\{E_{\mr_1}(u_i)\}_{i=1}^{\fz}$ has a limit at every point $z\in B(\boz^s, 1)\setminus(N\cup\mr_1(N))$. Define
\begin{equation}\label{equa:definev}
v(z):=\left\{\begin{array}{ll}
\lim_{i\to\fz}E_{\mr_1}(u_i)(z)&\ \  {\rm if}\ \ z\in B(\boz^s, 1)\setminus(N\cup\mr_1(N)),\\
0,&\ \  {\rm if}\ \ z\in N\cup\mr_1(N).
\end{array}\right.
\end{equation}
Since $\{u_i\}_{i=1}^\fz$ is a Cauchy sequence in $W^{1,p}(\boz^s)$, the inequalities (\ref{equa:limit0}) and (\ref{equa:normc}) yields that $\{E_{\mr_1}(u_i)\}_{i=1}^{\fz}$ is also a Cauchy sequence in $W^{1, q}(B(\boz^s, 1))$. Hence $v\in W^{1,q}(B(\boz^s, 1))$ with 
\[\|v\|_{W^{1, q}(B(\boz^s, 1))}\leq C\|u\|_{W^{1,p}(\boz^s)}.\]
By definition, we conclude that $E_{\mr_1}(u)(z)=v(z)$ for every $z\in B(\boz^s, 1)\setminus(N\cup\mr_1(N))$. Since $|N\cup\mr_1(N)|=0$, we have $E_{\mr_1}(u)\in W^{1,q}(B(\boz^s, 1))$ with
\[\|E_{\mr_1}(u)\|_{W^{1,q}(B(\boz^s, 1))}=\|v\|_{W^{1,q}(B(\boz^s, 1))}\leq C\|u\|_{W^{1,p}(\boz^s)}.\]

\end{proof}

\begin{thm}\label{thm6}
Let $\boz^s\subset\rr^n$ be an outward cuspidal domain with the degree $s>1$. Then the reflection $\mr_2:\widehat{\rr^n}\to\widehat{\rr^n}$ over $\partial\boz^s$ induces a bounded linear extension operator from $W^{1,p}(\boz^s)$ to $W^{1,q}(\rr^n)$, whenever $\frac{1+(n-1)s}{2+(n-2)s}<p<\fz$ and $1\leq q< \frac{(1+(n-1)s)p}{1+(n-1)s+(s-1)p}$.
\end{thm}
\begin{proof}

Let $u\in C_o^\fz(\rn)\cap W^{1,p}(\boz^s)$ be arbitrary. We define a function $E_{\mr_2}(u)$ as in (\ref{equa:E_r(u)}) and another function $w$ by setting 
\begin{equation}\label{defn:funcw'}
w(z):=\left\{\begin{array}{ll}
u\circ\mr_2(z),&\ \  {\rm if}\ \ z\in B(\boz^s, 1)\setminus\overline{\boz^s},\\
u(z),&\ \  {\rm if}\ \ z\in \overline{\boz^s}.
\end{array}\right.
\end{equation}
We claim that $w\in W^{1,q}(B(\boz^s, 1))$ with 
 \[\|w\|_{W^{1,q}(B(\boz^s, 1))}\leq C\|u\|_{W^{1,p}(\boz^s)}\]
 for a constant $C>1$ independent of $u$. As in the proof of Theorem \ref{thm5}, it suffices to estimate the norm over the union of
$\Omega^s$ and $B(\Omega^s,1)\setminus \overline{\boz^s}$ and we are again reduced to estimating the norm over the second set in question. On this set, $w=u\circ\mr_2$ almost everywhere and hence it suffices to prove the inequality
\begin{equation}\label{equa:Norm1}
\lf(\int_{B(\boz^s, 1)\setminus\overline{\boz^s}}|u\circ\mr_2(z)|^qdz\r)^{\frac{1}{q}}\leq C\lf(\int_{\boz^s}|u(z)|^pdz\r)^{\frac{1}{p}}
\end{equation}
and the inequality
\begin{equation}\label{equa:Norm2}
\lf(\int_{B(\boz^s, 1)\setminus\overline{\boz^s}}|D(u\circ\mr_2)(z)|^qdz\r)^{\frac{1}{q}}\leq C\lf(\int_{\boz^s}|Du(z)|^pdz\r)^{\frac{1}{p}}.
\end{equation}
Now
\[B(\boz^s, 1)\setminus\overline{\boz^s}=(B(\boz^s, 1)\setminus\Delta')\cup(\Delta'\setminus\overline{\boz^s})\]
and $\Delta'\setminus\overline{\boz^s}=D\cup E$. Since 
\[|\partial\Delta'|=|\partial D|=|\partial E|=0,\] 
we have 
\begin{eqnarray}\label{equa:Sum1}
\int_{B(\boz^s, 1)\setminus\overline{\boz^s}}|u\circ\mr_2(z)|^qdz&=&\int_{B(\boz^s, 1)\setminus\overline{\Delta'}}|u\circ\mr_2(z)|^qdz\\
       & &+\lf(\int_{\mathring D}+\int_{\mathring E}\r)|u\circ\mr_2(z)|^qdz.\nonumber
\end{eqnarray}
Since $\mr_2$ is bi-Lipschitz on $B(\boz^s, 1)\setminus\overline{\Delta'}$ and $|\boz^s|<\fz$, by the H\"older inequality, we have 
\begin{equation}\label{equa:Sum2}
\int_{B(\boz^s, 1)\setminus\overline{\Delta'}}|u\circ\mr_2(z)|^qdz\leq C\lf(\int_{\boz^s}|u(z)|^pdz\r)^{\frac{q}{p}}.
\end{equation}
Since $|J_{\mr_2}(t, x)|\sim1$ on $\mathring E\cup\mathring D$, by (\ref{equa:ref21}) and (\ref{equa:ref22}), we conclude by computing as in (\ref{akobi}) that
\begin{eqnarray}\label{equa:Sum3}
\int_{\mathring E\cup\mathring D}|u\circ\mr_2(z)|^qdz\leq C\lf(\int_{\boz^s}|u(z)|^pdz\r)^{\frac{q}{p}},
\end{eqnarray}
whenever $\frac{1+(n-1)s}{2+(n-2)s}<p<\fz$ and $1\leq q<\frac{(1+(n-1)s)p}{1+(n-1)s+(s-1)p}$. By combining inequalities (\ref{equa:Sum1})-(\ref{equa:Sum3}), we obtain  inequality (\ref{equa:Norm1}).

To prove inequality (\ref{equa:Norm2}), by Lemma \ref{QCcompo}, it suffices to show that
\begin{equation}
\int_{B(\boz^s, 1)\setminus\overline{\boz^s}}\frac{|D\mr_2(z)|^{\frac{pq}{p-q}}}{|J_{\mr_2}(z)|^{\frac{q}{p-q}}}dz<\fz.\nonumber
\end{equation}
Trivially, 
\begin{equation}
\int_{B(\boz^s, 1)\setminus\overline{\boz^s}}\frac{|D\mr_2(z)|^{\frac{pq}{p-q}}}{|J_{\mr_2}(z)|^{\frac{q}{p-q}}}dz=\int_{B(\boz^s, 1)\setminus\overline{\Delta'}}\frac{|D\mr_2(z)|^{\frac{pq}{p-q}}}{|J_{\mr_2}(z)|^{\frac{q}{p-q}}}dz+\int_{\Delta'\setminus\overline{\boz^s}}\frac{|D\mr_2(z)|^{\frac{pq}{p-q}}}{|J_{\mr_2}(z)|^{\frac{q}{p-q}}}dz.\nonumber
\end{equation}
Since $\mr_2$ is bi-Lipschitz on $B(\boz^s, 1)\setminus\overline{\Delta'}$, we have 
\begin{equation}
\int_{B(\boz^s, 1)\setminus\overline{\Delta'}}\frac{|D\mr_2(z)|^{\frac{pq}{p-q}}}{|J_{\mr_2}(z)|^{\frac{q}{p-q}}}dz<\fz.\nonumber
\end{equation}
Since $\Delta'\setminus\overline{\boz^s}= D\cup E$, $|\partial D|=|\partial E|=0$, inequalities (\ref{equa:ref21}), (\ref{equa:ref22}) give
\begin{eqnarray}\label{Summa3'}
\int_{\Delta'\setminus\overline{\boz^s}}\frac{|D\mr_2(z)|^{\frac{pq}{p-q}}}{|J_{\mr_2}(z)|^{\frac{q}{p-q}}}dz&\leq&\int_{\mathring D}\frac{|D\mr_2(z)|^{\frac{pq}{p-q}}}{|J_{\mr_2}(z)|^{\frac{q}{p-q}}} dz+\int_{\mathring E}\frac{|D\mr_2(z)|^{\frac{pq}{p-q}}}{|J_{\mr_2}(z)|^{\frac{q}{p-q}}} dz\\                                                                                                                                        &\leq& C\int_{0}^{\frac{1}{2}}\int_{t^{s}}^{\left(\frac{1}{2}\right)^{s}} |x|^{(n-2)-\frac{(s-1)pq}{s(p-q)}}d|x|dt+C\nonumber\\
									                                                                                  &\leq&C\int_{0}^{\frac{1}{2}}t^{(n-1)s-\frac{(s-1)pq}{p-q}}dt+C<\fz,\nonumber
\end{eqnarray}
whenever $\frac{1+(n-1)s}{2+(n-2)s}<p<\fz$ and $1\leq q<\frac{(1+(n-1)s)p}{1+(n-1)s+(s-1)p}$. In conclusion, we have proved that $w\in W^{1,q}(B(\boz^s, 1))$ with the bound
$$\|w\|_{W^{1,q}(B(\boz^s, 1))}\leq C\|u\|_{W^{1,p}(\boz^s)}$$
whenever $\frac{1+(n-1)s}{2+(n-2)s}<p<\fz$ and $1\leq q<\frac{(1+(n-1)s)p}{1+(n-1)s+(s-1)p}$. Since $E_{\mr_2}(u)=w$ almost everywhere, the above also holds with $w$ replaced by $E_{\mr_2}(u)$. Hence, we may complete the proof by following the argument of the proof of Theorem \ref{thm5}.

\end{proof}

\subsection{Proof of Theorem \ref{thm:comp}}
We begin with a useful observation.
\begin{lem}\label{lem:reftwo}
Let $1<s<\fz$ and $1<p<\fz$. If there is a reflection $\mr:\widehat{\rn}\to\widehat{\rn}$ over $\partial\boz^s$ which induces a bounded linear extension operator from $W^{1,p}(\rn\setminus\overline{\boz^s})$ to $W^{1,p}(\rn)$, then $\mr\in W^{1, p}_{\rm loc}(G\cap\boz^s, \rn)$ and 
$$|D\mr(z)|^p\leq C|J_\mr(z)|$$
for almost every $z\in G\cap\boz^s$, where $G$ is a bounded open set containing $\partial\boz^s$.
\end{lem}
\begin{proof}
Let $\mr:\widehat{\rn}\to\widehat{\rn}$ be a reflection over $\partial\boz^s$ which induces a bounded linear extension operator from $W^{1,p}(\rn\setminus\overline{\boz^s})$ to $W^{1,p}(\rn)$. Then there exists a bounded open set $U$ containing $\partial\boz^s$ so that the function 
\begin{equation}
E_{\mr}(u)(z):=\left\{\begin{array}{ll}\label{equa:T_r(u)}
u\circ\mr(z),&\ {\rm for}\ z\in U\cap\boz^s,\\
0,&\ {\rm for}\ z\in\partial\boz^s,\\
u(z),&\ {\rm for}\ z\in U\setminus\overline{\boz^s}
\end{array}\right.
\end{equation}
belongs to $W^{1,p}(U)$ and satisfies
$$\|E_\mr(u)\|_{W^{1,p}(U)}\leq C\|u\|_{W^{1,p}(U\setminus\boz^s)}$$
for a positive constant $C$ independent of $u$. It follows that $\mr\in W^{1, p}_{\rm loc}(U\cap\boz^s, \rn)$. We employ an idea from \cite{DP} and pick a Lipschitz domain $G$ so that $\overline{\boz^s}\subset G$ and $\partial G\subset U$. Since $G$ is Lipschitz and contains the closure of $\boz^s$, the geometry of $\boz^s$ easily yields that $G\setminus\overline{\boz^s}$ is an $(\epsilon, \delta)$-domain for some positive $\epsilon, \delta$. Since $u-u_{G\setminus\overline{\boz^s}}\in W^{1,p}(G\setminus\overline{\boz^s})$ and $(\epsilon, \delta)$-domains are $(p, p)$-extension domains, we find a function $v\in W^{1,p}(\rn\setminus\overline{\boz^s})$ such that $v=u-u_{G\setminus\overline{\boz^s}}$ on $G\setminus\overline{\boz^s}$ and 
\begin{equation}\label{palautus0}
\|v\|_{W^{1,p}(\rn\setminus\overline{\boz^s})}\leq C\|u-u_{G\setminus\overline{\boz^s}}\|_{W^{1,p}(G\setminus\overline{\boz^s})}.
\end{equation}
Next, since $G\setminus\overline{\boz^s}$ is a bounded $(\epsilon, \delta)$-domain,  we have
\begin{equation}\label{palautus2}
\int_{G\setminus\overline{\boz^s}}|u(z)-u_{G\setminus\overline{\boz^s}}|^pdz\leq C\int_{G\setminus\overline{\boz^s}}|Du(z)|^pdz,
\end{equation}
see \cite{BKmrt, SStams}. By our assumption, (\ref{palautus0}) and (\ref{palautus2}), we have 
 \begin{eqnarray}\label{palautus1}
 \|v\circ \mr\|_{W^{1,p}(G\cap\boz^s)}&\leq& \|v\circ\mr\|_{W^{1,p}(U\cap\boz^s)}\nonumber\\
                                                      &\leq& C\|u-u_{G\setminus\overline{\boz^s}}\|_{W^{1,p}(G\setminus\overline{\boz^s})}\leq C\|Du\|_{L^p(G\setminus\overline{\boz^s})}.\nonumber
 \end{eqnarray}
 It is easy to check that $v\circ\mr=E_\mr(v)$ on $G\cap\boz^s$ and that $Du=Dv$ almost everywhere on $G\setminus\overline{\boz^s}$. Hence, we have
\begin{eqnarray}\label{equa:Lpex}
\int_{G\cap\boz^s}|DE_\mr(v)(z)|^pdz\leq C\int_{G\setminus\overline{\boz^s}}|Du(z)|^pdz.\nonumber
\end{eqnarray}
Since $u\in W^{1,p}(\boz^s)$ is arbitrary, Lemma \ref{QCcompo} gives the asserted inequality. 
\end{proof}

We are now ready to prove Theorem \ref{thm:comp}.

\begin{proof}[Proof of Theorem \ref{thm:comp}]
Fix $1<s<\fz$. It is easy to check that $\rr^n\setminus\overline{\boz^s}$ is an $(\epsilon,\delta)$-domain, for some positive constants $\epsilon$ and $\delta$. Hence, by \cite{Jones}, $\rr^n\setminus\overline{\boz^s}$ is a $(p, p)$-extension domain, for every $p\in[1, \fz)$.

We begin by showing that the reflection $\mr_1$ induces a bounded linear extension operator from $W^{1,p}(\rn\setminus\overline{\boz^s})$ to $W^{1,p}(\rn)$, whenever $1\leq p\leq n-1$. Define the domain $\Delta$ as in (\ref{equa:delta}) and the domain $\boz^s_1$ as in (\ref{subdomain}).
By (\ref{HOMEOMOR'}), the formula of the reflection $\mr_1$ on $\boz^s_1$ is
\begin{equation}\label{REFLEC}
\mr_1(t, x)=\left\{\begin{array}{ll}
\lf(-t, \frac{6x}{t^{s-1}}\r),&\ \  {\rm if}\ \ 0\leq |x|<\frac{1}{6}t^{s},\\
\lf(\frac{12|x|}{t^{s-1}}-3t, t\frac{x}{|x|}\r),&\ \ {\rm if}\ \ \frac{1}{6}t^{s}\leq |x|<\frac{1}{3}t^{s},\\
\lf(t, \frac{3(t^s-t)}{2t^s}x+\lf(\frac{3t}{2}-\frac{t^s}{2}\r)\frac{x}{|x|}\r),&\ \  {\rm if}\ \ \frac{1}{3}t^{s}\leq |x|<t^{s}.\\
\end{array}\right.
\end{equation}
For every $(t, x)\in \boz^s_1$ with $0< |x|<\frac{1}{6}t^{s}$, the resulting differential matrix of $\mr_1$ is 
\begin{eqnarray}
D_{\mr_1}(t, x)
&=&
\left(
 \begin{array}{ccccc}
-1 &~~ 0 &~~ 0 &~~ \cdots&~~ 0\\
(1-s)\frac{6x_1}{t^s}  &~~ \frac{6}{t^{s-1}}&~~ 0 &~~ \cdots&~~0\\
(1-s)\frac{6x_2}{t^s} &~~ 0 &~~ \frac{6}{t^{s-1}}  &~~\cdots &~~0\\
\vdots &~~\vdots &~~\vdots &~~\ddots &~~\vdots\\
(1-s)\frac{6x_{n-1}}{t^s} &~~ 0 &~~ 0 &~~\cdots &~~\frac{6}{t^{s-1}}
\end{array}
\right).  \nonumber
\end{eqnarray}
After a simple computation, for every $(t, x)\in\boz^s_1$ with $0<|x|<\frac{1}{6}t^s$, we have 
\begin{equation}\label{Distor1}
|D\mr_1(t, x)|=\frac{6}{t^{s-1}}\ {\rm and}\ |J_{\mr_1}(t, x)|=\lf(\frac{6}{t^{s-1}}\r)^{n-1}.
\end{equation}

For every $(t, x)\in \boz^s_1$ with $\frac{1}{6}t^{s}< |x|<\frac{1}{3}t^{s}$, the resulting differential matrix is 
\begin{eqnarray}\label{DifferA2}
D\mr_1(t, x)
&=&
\left(
 \begin{array}{ccccc}
\frac{12(1-s)|x|}{t^s}-3 &~~ \frac{12x_1}{|x|t^{s-1}} &~~ \frac{12x_2}{|x|t^{s-1}}&~~\cdots &~~ \frac{12x_{n-1}}{|x|t^{s-1}}\\
\frac{x_1}{|x|}  &~~ A^1_1(t, x)&~~ A^1_2(t, x) &~~\cdots&~~ A^1_{n-1}(t, x)\\
\frac{x_2}{|x|} &~~ A^2_1(t, x) &~~ A^2_2(t, x) &~~ \cdots &~~ A^2_{n-1}(t, x)\\
\vdots &~~\vdots &~~ \vdots &~~\ddots &~~\vdots\\
\frac{x_{n-1}}{|x|} &~~ A^{n-1}_1(t, x) &~~ A^{n-1}_2(t, x) &~~\cdots&~~ A^{n-1}_{n-1}(t, x)
\end{array}
\right),  \nonumber
\end{eqnarray}
where, for every $i, j\in\{1, 2, \cdots, n-1\}$, we have
\begin{equation}
A^i_j(t, x):=\left\{\begin{array}{ll}
\frac{t}{|x|}-\frac{tx_i^2}{|x|^3},&\ \  {\rm if}\ \ i=j,\\
\frac{-tx_ix_j}{|x|^3},&\ \ {\rm if}\ \ i\neq j.
\end{array}\right.\nonumber
\end{equation}
After a simple computation, for every $(t, x)\in\boz^s_1$ with $\frac{1}{6}t^s<|x|<\frac{1}{3}t^s$, we have
\begin{equation}\label{Distor2}
|D\mr_1(t, x)|\leq \frac{C}{t^{s-1}} \ {\rm and}\ |J_{\mr_1}(t, x)|\sim_c \lf(\frac{1}{t^{s-1}}\r)^{n-1}.
\end{equation}

 For every $(t, x)\in \boz^s_1$ with $\frac{1}{3}t^{s}<|x|<t^{s}$, the resulting differential matrix is 
\begin{eqnarray}
D\mr_1(t, x)
&=&
\left(
 \begin{array}{ccccc}
1 &~~ 0 &~~ 0 &~~\cdots&~~0\\
A^1_t(t, x)  &~~ A^1_1(t, x)&~~ A^1_2(t, x) &~~\cdots&~~A^1_{n-1}(t, x)\\
A^2_t(t, x) &~~ A^2_1(t, x) &~~ A^2_2(t, x) &~~\cdots&~~A^2_{n-1}(t, x)\\
\vdots &~~\vdots&~~\vdots&~~\ddots&~~\vdots\\
A^{n-1}_t(t, x) &~~ A^{n-1}_1(t, x)&~~ A^{n-1}_2(t, x) &~~\cdots&~~ A^{n-1}_{n-1}(t, x)
\end{array}\nonumber
\right), \\ \nonumber
\end{eqnarray}
where, for every $i, j\in \{1, 2, \cdots, n-1\}$, we have 
\begin{equation}
A^i_j(t, x):=\left\{\begin{array}{ll}
\lf(\frac{3}{2}-\frac{3}{2t^{s-1}}\r)+\lf(\frac{3t}{2}-\frac{t^s}{2}\r)\lf(\frac{1}{|x|}-\frac{x_i^2}{|x|^3}\r),&\ \  {\rm if}\ \ i=j,\\
-\lf(\frac{3t}{2}-\frac{t^s}{2}\r)\frac{x_ix_j}{|x|^3},&\ \ {\rm if}\ \ i\neq j.
\end{array}\right.\nonumber
\end{equation}
and
\begin{equation}
A^i_t(t, x):=(s-1)\frac{3x_i}{2t^s}+\lf(\frac{3}{2}-\frac{s}{2}t^{s-1}\r)\frac{x_i}{|x|}.\nonumber
\end{equation}
After a simple computation, for every $(t, x)\in\boz^s_1$ with $\frac{1}{3}t^s<|x|<t^s$, we have
\begin{equation}\label{Distor3}
|D\mr_1(t, x)|\leq \frac{C}{t^{s-1}}\ {\rm and}\ |J_{\mr_1}(t, x)|\sim_c\lf(\frac{1}{t^{s-1}}\r)^{n-1}.
\end{equation}

By combining (\ref{Distor1}), (\ref{Distor2}) and (\ref{Distor3}), we conclude that 
\begin{equation}\label{equa:pdist}
|D\mr_1(z)|^p\leq C|J_{\mr_1}(z)|
\end{equation}
for almost every $z\in \Delta\cap\boz^s$.
By the same inequalities, since $\mr_1$ is locally bi-Lipschitz on $B(\boz^s, 1)\setminus\overline\Delta$, for every $u\in C_o^\fz(\rn)\cap W^{1,p}(\rn\setminus\overline{\boz^s})$,  we have 
\begin{eqnarray}\label{equa:LpF}
\int_{\mr_1(B(\boz^s, 1)\setminus\overline{\boz^s})}|u\circ\mr_1(z)|^pdz&\leq&C\int_{\mr_1(B(\boz^s, 1)\setminus\overline{\boz^s})}|u\circ\mr_1(z)|^p|J_{\mr_1}(z)|dz\\
       &\leq&\int_{B(\boz^s, 1)\setminus\overline{\boz^s}}|u(z)|^pdz.\nonumber
\end{eqnarray}
Moreover, by Lemma \ref{QCcompo} and (\ref{equa:pdist}), we have 
\begin{equation}\label{equa:LpD}
\int_{B(\boz^s, 1)\setminus\overline{\boz^s}}|D(u\circ\mr_1)(z)|^pdz\leq \int_{\rn\setminus\overline{\boz^s}}|Du(z)|^pdz.
\end{equation}
Since $\rn\setminus\overline{\boz^s}$ satisfies the segment condition, (\ref{equa:LpF}) and (\ref{equa:LpD}) allow us to repeat the argument in the proof of Theorem \ref{thm5} so as to conclude that $\mr_1$ induces a bounded linear extension operator from $W^{1, p}(\rn\setminus\overline{\boz^s})$ to $W^{1,p}(\rn)$, whenever $1\leq p\leq n-1$.

Next, we show that there is no reflection over $\partial\boz^s$ which can induce a bounded linear extension operator from $W^{1,p}(\rn\setminus\overline{\boz^s})$ to $W^{1,p}(\rn)$, for any $n-1<p<\fz$. Let $n-1<p<\infty$ be fixed. Suppose that there exists a reflection $\mr:\widehat\rr^n\rightarrow\widehat\rr^n$ over $\partial\boz^s$, which induces a bounded linear extension operator from $W^{1,p}(\rr^n\setminus\overline{\boz^s})$ to $W^{1,p}(\rr^n)$. By Lemma \ref{lem:reftwo}, there exists an open set $G$ which contains $\partial\boz^s$ such that for almost every $z\in G\cap\boz^s$,  we have 
\begin{equation}
|D\mr(z|^p\leq C|J_\mr(z)|.\nonumber
\end{equation}
Then, by Lemma \ref{reduinverse}, for almost every $(t, x)\in \mr\lf(G\cap\boz^s\r)$, we have 
\begin{equation}\label{ineq:inver}
|D\mr(z)|^{\frac{p}{p+1-n}}\leq C|J_{\mr}(z)|.
\end{equation}
Let  $u\in C_o^{\fz}(\rn)\cap W^{1,p}(\boz^s)$ be arbitrary. By definition, $E_\mr(u)$ is bounded and continuous on $G$. Pick a Lipschitz domain $\widetilde G$ so that $\overline{\boz^s}\subset G$ and $\partial\widetilde G\subset G$. By Lemma \ref{QCcompo}, we have 
 \begin{equation}\label{equa:HNC}
 \|DE_{\mr}(u)\|_{L^{\frac{p}{p+1-n}}(\widetilde G)}\leq C\|Du\|_{L^{\frac{p}{p+1-n}}(\boz^s)}.
 \end{equation}
We conclude that $E_\mr(u)\in W^{1,\frac{p}{p+1-n}}(\widetilde G)$. Since $\widetilde G$ is a Lipschitz domain, \cite[Lemma 4.1]{DP} implies 
\begin{equation}\label{equa:POIN}
\int_{\widetilde G}|E_\mr(u)(z)-u_{\boz^s}|^{\frac{p}{p+1-n}}dz\leq C(\widetilde{G}, \boz^s)\int_{\widetilde G}|DE_\mr(u)(z)|^{\frac{p}{p+1-n}}dz.
\end{equation} 
Hence, we have 
\begin{equation}\label{equa:HNC1}
\|E_\mr(u)\|_{L^{\frac{p}{p+1-n}}(\widetilde G)}\leq C\lf(\|DE_\mr(u)\|_{L^{\frac{p}{p+1-n}}(\widetilde G)}+\|u\|_{L^{\frac{p}{p+1-n}}(\boz^s)}\r).
\end{equation}
By combining inequalities (\ref{equa:HNC}) and (\ref{equa:HNC1}), we obtain 
\begin{equation}
\|E_\mr(u)\|_{W^{1,\frac{p}{p+1-n}}(\widetilde G)}\leq \|u\|_{W^{1, \frac{p}{p+1-n}}(\boz^s)}.
\end{equation}
 
Since $C_o^{\fz}(\rn)\cap W^{1, \frac{p}{p+1-n}}(\boz^s)$ is dense in $W^{1, \frac{p}{p+1-n}}(\boz^s)$, for every function $u\in W^{1,\frac{p}{p+1-n}}(\boz^s)$, there exists a sequence of functions $u_i\in C_o^{\fz}(\rn)\cap W^{1,\frac{p}{p+1-n}}(\boz^s)$ such that
\begin{equation}\label{equa:appro}
\lim_{i\to\fz}\|u_i-u\|_{W^{1, \frac{p}{p+1-n}}(\boz^s)}=0,
\end{equation}
and for almost every $z\in\boz^s$, 
$$\lim_{i\to\fz}|u_i(z)-u(z)|=0.$$
By (\ref{equa:HNC}) and (\ref{equa:appro}), $\{E_\mr(u_i)\}_{i=1}^{\fz}$ is a Cauchy sequence in $W^{1,\frac{p}{p+1-n}}(\widetilde G)$. By the completeness of $W^{1, \frac{p}{p+1-n}}(\widetilde G)$, there exits a function $\omega\in W^{1,\frac{p}{p+1-n}}(\widetilde G)$ with 
$$\lim_{i\to\fz}\|w-E_\mr(u_i)\|_{W^{1, \frac{p}{p+1-n}}(\widetilde G)}=0$$
and $\omega(z)=u(z)$ for almost every $z\in\boz^s$. We define $E_\mr(u)(z):=\omega(z)$ on $\widetilde G$. By (\ref{equa:HNC}) and (\ref{equa:appro}) again, we have 
$$\|E_\mr(u)\|_{W^{1,\frac{p}{p+1-n}}(\widetilde G)}\leq C\|u\|_{W^{1,\frac{p}{p+1-n}}(\boz^s)}.$$
Hence, $\boz^s$ is a Sobolev $\lf(\frac{p}{p+1-n}, \frac{p}{p+1-n}\r)$-extension domain. This contradicts  the classical result that $\boz^s$ is not a $(q, q)$-extension domain, for any $1\leq q<\fz$, see \cite{Mazya} and references therein.
\end{proof}

\subsection{ Proof of Theorem \ref{thm:infty}}

\begin{proof}[Proof of Theorem \ref{thm:infty}]
Fix $1<s<\fz$. It is easy to see both $\boz^s$ and $\rr^n\setminus\overline{\boz^s}$ are uniformly locally quasiconvex. By \cite{HKT}, they are $(\fz,\fz)$-extension domains.

To begin, we show that the reflection $\mr_1$ induces a bounded linear extension operator from $W^{1,\fz}(\boz^s)$ to $W^{1, \fz}(\rn)$. Since $\boz^s$ is uniformly quasiconvex, every function in $W^{1, \fz}(\boz^s)$ has a Lipschitz representative. Without loss of generality, we assume every function in $W^{1,\fz}(\boz^s)$ is Lipschitz. Let $u\in W^{1, \fz}(\boz^s)$ be arbitrary. Define the extension $E_{\mr_1}(u)$ on $B(\boz^s, 1)$ as in (\ref{equa:E_r(u)}). Since $u\in W^{1,\fz}(\boz^s)$ is Lipschitz and $\mr_1$ is locally Lipschitz on $B(\boz^s, 1)\setminus(\boz^s\cup\{0\})$, we have $E_{\mr_1}(u)\in W^{1, 1}_{\rm loc}(B(\boz^s, 1)\setminus\overline{\boz^s})$. By (\ref{equa:ref11}), (\ref{equa:ref12}), (\ref{equa:ref13}) and the fact that $\mr_1$ is bi-Lipschitz on $B(\boz^s, 1)\setminus\overline{\Delta}$, for almost every $z\in B(\boz^s, 1)\setminus\overline{\boz^s}$, we have 
\begin{equation}
|DE_{\mr_1}(u)(z)|\leq C|Du(\mr_1(z))|.\nonumber
\end{equation}
This implies that
\[\|E_{\mr_1}(u)\|_{W^{1, \fz}(B(\boz^s, 1))}\leq C\|u\|_{W^{1,\fz}(\boz^s)}\]
as desired.

Next, we show that there does not exist a reflection over $\partial\boz^s$ which can induce a bounded linear extension operator from $W^{1, \fz}(\rn\setminus\overline{\boz^s})$ to $W^{1,\fz}(\rn)$. Define a function $u\in W^{1, \fz}(\rr^n\setminus\overline{\boz^s})$ by setting 
\begin{equation}\label{REFLEC}
u(t, x)=\left\{\begin{array}{ll}
1,&\ \  {\rm if}\ \ (t, x)\in\rr^n\setminus\overline{\boz^s}\ {\rm and}\ t\geq 1,\\
t,&\ \ {\rm if}\ \ (t, x)\in\rr^n\setminus\overline{\boz^s}\ {\rm and}\ 0<t<1,\\
0,&\ \  {\rm if}\ \ (t, x)\in\rr^n\setminus\overline{\boz^s}\ {\rm and}\ t\leq 0.\\
\end{array}\right.
\end{equation}
For every $t\in(0,1)$ fixed, we define a $2$-dimensional disk $D_t\subset\boz^s$ by setting
\begin{equation}
D_t:=\lf\{(t, x)\in\rr^n; |x|<t^s\r\}\nonumber
\end{equation}
 and define 
 \begin{center}
 $S_t:=\lf\{(t, x)\in\rr^n; |x|=2t^{s}\r\}$.
 \end{center}
 Suppose to the contrary that there exists a reflection $\mr:\widehat{\rr^n}\to\widehat{\rr^n}$ over $\partial\boz^s$ which induces a bounded linear extension operator from $W^{1, \fz}(\rr^n\setminus\overline{\boz^s})$ to $W^{1, \fz}(\rr^n)$. Define the function $E_\mr(u)$ on $B(\boz^s, 1)$ as in (\ref{equa:E_r(u)}). By the geometry of $\boz^s$ and the fact that $\mr$ is continuous and $\mr(z)=z$ whenever $z\in\partial\boz^s$, there exists a small enough $t_o\in(0,1)$ such that for every $t\in(0, t_o)$, there exists $(t, x_t)\in D_t$ with $E_\mr(u)((t, x_t))=0$ and there exists $(t, x'_t)\in S_t$ with $E_\mr(u)((t, x'_t))=t$ and $d((t, x_t), (t, x'_t))\leq 2t^s$. Hence for every $0<t<t_o$, we have 
 \begin{equation}
 |E_\mr(u)((t, x_t))-E_\mr(u)((t, x'_t))|\geq t\geq Cd^{\frac{1}{1+s}}((t, x_t), (t, x'_t)).\nonumber
 \end{equation}
This contradicts the assumption that $E_\mr(u)\in W^{1,\fz}(B(\boz_s, 1)):$ since $B(\boz^s, 1)$ is uniformly locally quasiconvex, a function in $W^{1,\fz}(B(\boz^s, 1))$ must have a Lipschitz representative.
\end{proof}

\noindent Pekka Koskela\\
  Zheng Zhu

\medskip

\noindent
Department of Mathematics and Statistics, University of Jyv\"askyl\"a, FI-40014, Finland.

\noindent{\it E-mail address}:  \texttt{pekka.j.koskela@jyu.fi}\\
                                 \texttt{zheng.z.zhu@jyu.fi}
\end{document}